\documentclass[10pt,leqno]{article}

\usepackage{amsmath,amssymb,amsthm,mathrsfs,dsfont}

\usepackage[margin=3cm]{geometry} 

\usepackage{titlesec,hyperref}

\usepackage{color}

\usepackage{fancyhdr}
\titleformat{\subsection}{\it}{\thesubsection.\enspace}{1pt}{}

\newtheorem{theo}{Theorem}[section]
\newtheorem{lemm}[theo]{Lemma}

\newtheorem{prop}[theo]{Proposition}
\newtheorem{rema}[theo]{Remark}
\numberwithin{equation}{section}

\allowdisplaybreaks 

\begin{document}
\title{Optimal decay rate for the generalized Oldroyd-B model with
	only stress tensor diffusion in $\mathbb{R}^2$
\hspace{-4mm}
}

\author{ Zhaonan $\mbox{Luo}^1$ \footnote{E-mail: luozhn@fudan.edu.cn},\quad
Wei $\mbox{Luo}^2$\footnote{E-mail:  luowei23@mail2.sysu.edu.cn} \quad and\quad
Zhaoyang $\mbox{Yin}^{2,3}$\footnote{E-mail: mcsyzy@mail.sysu.edu.cn}\\
$^1\mbox{School}$ of Mathematical Sciences, Fudan University, Shanghai 200433, China.\\
$^2\mbox{Department}$ of Mathematics,
Sun Yat-sen University, Guangzhou 510275, China\\
$^3\mbox{Faculty}$ of Information Technology,\\ Macau University of Science and Technology, Macau, China}

\date{}
\maketitle
\hrule

\begin{abstract}
In this paper, we are concerned with optimal decay rate for the 2-D generalized Oldroyd-B model with
only stress tensor diffusion $(-\Delta)^{\beta}\tau$. In the case $\beta=1$, we first establish optimal decay rate in $H^1$ framework and remove the smallness assumption of low frequencies by virtue of the Fourier splitting method and the Littlewood-Paley decomposition theory. Furthermore, we prove optimal decay rate for the highest derivative of the solution by a different method combining time frequency decomposition and the time weighted energy estimate. In the case $\frac 1 2\leq \beta <1$, we study optimal decay rate for the highest derivative of the solution by the improved Fourier splitting method. \\
\vspace*{5pt}

\vspace*{5pt}
\end{abstract}

\vspace*{10pt}

\tableofcontents

\section{Introduction}
The classical Oldroyd-B model can be written as follows:
\begin{align}\label{eq0}
\left\{\begin{array}{l}
\partial_tu + u\cdot\nabla u+\nabla P = div~\tau+\nu\Delta u,~~~~div~u=0,\\[1ex]
\partial_t\tau + u\cdot\nabla\tau+a_0\tau+Q(\nabla u,\tau)=\alpha D(u),\\[1ex]
u|_{t=0}=u_0,~~\tau|_{t=0}=\tau_0. \\[1ex]
\end{array}\right.
\end{align}
The Oldroyd-B model derived by J. G. Oldroyd in \cite{1958Non} which was used to describe the dynamics of viscoelastic fluids.
In \eqref{eq0}, $u(t,x)$ represents the velocity of the liquid, $P$ denotes the pressure and $\tau(t,x)$ is the symmetric stress tensor. The parameters $a_0$, $\nu$ and $\alpha$ are nonnegative.
Furthermore, $Q$ is the following bilinear form
$$Q(\nabla u, \tau)=\tau \Omega-\Omega\tau+b(D(u)\tau+\tau D(u)),~~~~b\in[-1, 1],$$
with the deformation tensor $D(u)=\frac {\nabla u+(\nabla u)^T} {2}$ and the vorticity tensor $\Omega=\frac {\nabla u-(\nabla u)^T} {2}$.
One can refer to \cite{2015Elgindi} for more explanations of the Oldroyd-B model.

In this paper, we are concerned with the generalized Oldroyd-B model with
only stress tensor diffusion and without the damping $(a_0=0)$:
\begin{align}\label{eq1}
\left\{\begin{array}{l}
\partial_tu + u\cdot\nabla u+\nabla P = div~\tau,~~~~div~u=0,\\[1ex]
\partial_t\tau + u\cdot\nabla\tau+Q(\nabla u,\tau)+(-\Delta)^{\beta}\tau= D(u),\\[1ex]
u|_{t=0}=u_0,~~\tau|_{t=0}=\tau_0. \\[1ex]
\end{array}\right.
\end{align}
Taking $\tau=0$ in $\eqref{eq1}$, then we obtain $Du=0$, which means that $u=0$ in Sobolev spaces. This observation reveals the essential difference between \eqref{eq1} and the well-known Euler equation.

\subsection{The Oldroyd-B model}
We first recall some mathematic results for the classical Oldroyd-B model \eqref{eq0}. In \cite{Guillope1990}, C. Guillop\'e,  and J. C. Saut first proved that the Oldroyd-B model admits a unique global strong solution in Sobolev spaces. The $L^p$-setting was showed by E. Fern\'andez-Cara, F.Guill\'en and R. Ortega \cite{Fernandez-Cara}. The global week solutions of the Oldroyd-B model were obtained by P. L. Lions and N. Masmoudi \cite{Lions-Masmoudi} for the co-rotation case that is $b=0$. However, the problem for the general case $b\neq0$ is still open, see \cite{2011Global,Masmoudi2013}. J. Y. Chemin and N. Masmoudi \cite{Chemin2001}
showed the existence and uniqueness of strong solutions in homogenous Besov spaces with critical index of regularity. Optimal decay rates for solutions to the 3-D Oldroyd-B model were given by M. Hieber, H. Wen and R. Zi \cite{2019OldroydB}. An approach based on the deformation tensor can be found in \cite{Li1,Li2,Lei-Zhou2005,Lei2008,2010On,Zhang-Fang2012,Cai2019}.

Some mathematic results for the generalized Oldroyd-B model \eqref{eq1} are given as follows. T. M. Elgindi and F. Rousset \cite{2015Elgindi} first showed global regularity for \eqref{eq1} with $a_0>0$ and $\beta=1$. In \cite{2015Elgindi1}, T. M. Elgindi and J. Liu  proved global strong solutions of the 3-D case under the assumption that initial data is sufficiently small. Recently, W. Deng, Z. Luo and Z. Yin \cite{DLY} obtained the global solutions in co-rotation case and proved the $H^1$ decay rate for the global solutions to \eqref{eq1}.
For the case $a_0=0$ and $\beta\in[\frac 1 2,1]$, P. Constantin, J. Wu, J. Zhao and Y. Zhu \cite{P.Constantin} established the global well-posedness for \eqref{eq1} with small data. In \cite{Wu}, J. Wu and J. Zhao investigated the global well-posedness in Besov spaces for fractional dissipation with small data. Optimal time decay rate in $H^1$ framework of global solutions to \eqref{eq1} was given by \cite{Wu1,Liushuai}.
However, they can't deal with critical case $d=2$ and $\beta=1$.

\subsection{Main results}
Optimal decay rate for \eqref{eq1} has been studied widely. The problem of time decay rate with $d=2$ is more difficult than the case with $d\geq 3$. Since the additional stress tensor $\tau$ does not decay fast enough, we failed to use the bootstrap argument as in \cite{Schonbek1985,Luo-Yin2}. For critical case $d=2$, we can not get any algebraic decay rate by the Fourier splitting method directly as in \cite{He2009}. We will study this problem by virtue of the refinement to Schonbek's \cite{Schonbek} strategy. Using the Fourier splitting method, we obtain initial logarithmic decay rate   \begin{align*}
\|(u,\tau)\|_{L^2}\leq C\ln^{-l}(e+t),
\end{align*}
for any $l\in N^{+}$. The main difficulty for proving optimal decay rate is lack of damping term and low-frequency estimation of $L^1$ for $\tau$. However, we have
\begin{align*}
\int_{S(t)}\int_{0}^{t}|\mathcal{F}Q(\nabla u, \tau)\cdot\bar{\hat{\tau}}|ds'd\xi\leq C(1+t)^{-\frac 1 2} \int_{0}^{t}\|\tau\|^2_{L^{2}}\|\nabla u\|_{L^{2}}ds'.
\end{align*}
By virtue of the time weighted energy estimate and the logarithmic decay rate, then we improve the time decay rate to algebraic decay rate $(1+t)^{-\frac{1}{2}}$. Notice that this is not the optimal decay rate.
Then, we prove optimal decay rate in $H^1$ framework for the 2-D generalized Oldroyd-B type model by virtue of the Fourier spiltting method and the Littlewood-Paley decomposition theory.
Considering the decay rate for the highest derivative of the solution to \eqref{eq1}, the main difficulty is unclosed energy estimate. The complete dissipation $\|\nabla\Lambda^s\tau\|^2_{L^2}+\|\nabla\Lambda^{s-1} u\|^2_{L^2}$ can be obtained by estimating the mixed term $\|\Lambda^s(u,\tau)\|^2_{L^2}+\langle \Lambda^{s-1}\tau, -\nabla\Lambda^{s-1} u\rangle$. One can see that the decay rate of inner product is slower than the decay rate of energy.
To overcome this difficulty, we construct a different energy and dissipation functional for $(u,\tau)$ as follows:
$$\widetilde{E}_s=k(1+t)\|\Lambda^s(u,\tau)\|^2_{L^2}+\langle \Lambda^{s-1}\tau, -\nabla\Lambda^{s-1} u\rangle,$$
and
$$\widetilde{D}_s=k(1+t)\|\nabla\Lambda^s\tau\|^2_{L^2}+ \frac 1 4\|\nabla\Lambda^{s-1} u\|^2_{L^2},$$
where $k$ is small enough. Finally, we introduce a new method which flexibly combines the Fourier splitting method and the time weighted energy estimate to prove the decay rate for the highest derivative.

We also present optimal decay rate for the generalized Oldroyd-B model \eqref{eq1} in the fractional case $\frac 1 2\leq\beta<1$.
By virtue of the traditional Fourier splitting method, one can not obtain the optimal decay for the fractional case. P. Wang, J. Wu, X. Xu and Y. Zhong \cite{Wu1}
use spectral analysis method to prove the optimal decay rate of the lower order energy of the fractional case. The corresponding linearized system of \eqref{eq1} is given by
\begin{align}\label{P1}
\left\{\begin{array}{l}
\partial_tu=\mathbb{P}div~\tau,~~~~div~u=0,\\[1ex]
\partial_t\mathbb{P}div~\tau+(-\Delta)^{\beta}\mathbb{P}div~\tau= \Delta u. \\[1ex]
\end{array}\right.
\end{align}
where $\mathbb{P}$ denotes the Leray projection onto divergence free vector fields. One can see that $(u,\mathbb{P}div~\tau)$ satisfy the same system of wave-type equation
\begin{align}\label{P2}
\partial_{tt}W+(-\Delta)^{\beta}\partial_{t}W-\Delta W=0.
\end{align}
The structure in \eqref{P2} reveals that there are both dissipative and dispersive effects
on $(u,\mathbb{P}div~\tau)$. However,
we point out that the refinement of Schonbek's strategy also work in this case with extra lower energy dissipation estimate
\begin{align}\label{P3}
\frac{d}{dt}[\|(u,\tau)\|^2_{H^s}+2k(\langle-\nabla u,\tau\rangle_{H^{s-\beta}}+\langle \tau, -\Lambda^{-2+2\beta}\nabla u\rangle)] + \frac k 2\|\Lambda^{\beta} u\|^2_{H^{s+1-2\beta}} + \|\Lambda^{\beta}\tau\|^2_{H^s} \leq 0,
\end{align}
where we use the properties of Calderon-Zygmund type operator. Then, we rediscover optimal decay rates in $H^1$ \cite{Wu1} by the improved Fourier splitting method. The main difficulty to obtain optimal decay rate for the highest derivative of the solution is unclosed energy estimate.
The complete dissipation $\|\Lambda^{s+\beta}\tau\|^2_{L^2}+\|\nabla\Lambda^{s-\beta} u\|^2_{L^2}$ can be obtained by estimating the mixed term $\|\Lambda^s(u,\tau)\|^2_{L^2}+\langle \Lambda^{s-\beta}\tau, -\nabla\Lambda^{s-\beta} u\rangle$. One can see that the decay rate of inner product is slower than the decay rate of energy.
To overcome this difficulty, we construct a different energy and dissipation functional for $(u,\tau)$ as follows:
$$\overline{E}_{\beta}=(1+t)^a\|\Lambda^s(u,\tau)\|^2_{L^2}+k\langle \Lambda^{s-\beta}\tau, -\nabla\Lambda^{s-\beta} u\rangle,$$
and
$$\overline{D}_{\beta}=(1+t)^a\|\Lambda^{s+\beta}\tau\|^2_{L^2}+ \frac k 4\|\nabla\Lambda^{s-\beta} u\|^2_{L^2},$$
where $a=2-\frac 1 {\beta}\in[0,1)$ and $k$ is small enough.
We prove optimal decay rate for the highest derivative of the solution to 2-D generalized Oldroyd-B type model by the improved Fourier splitting method. Notice that this novel result for \eqref{eq1} has not been studied before.

Let's review the global existence result for \eqref{eq1}.
\begin{theo}\cite{P.Constantin}\label{th1}
	Let $d=2$, $\frac 1 2\leq \beta \leq 1$ and $s>2$. Let $(u,\tau)$ be a strong solution of \eqref{eq1} with the initial data $(u_0,\tau_0)\in H^s$ and $\tau_0$ is symmetric. There exists a small constant $\delta$ such that if
	\begin{align*}
	\|(u_0,\tau_0)\|_{H^s}\leq \delta,
	\end{align*}
	then the system \eqref{eq1} admits a unique global strong solution $(u,\tau)\in C([0,\infty); H^s)$. Moreover, the energy estimate for $(u,\tau)$ implies that
	\begin{align}\label{estimate}
	\frac{d}{dt}(\|(u,\tau)\|^2_{H^s}+2k\langle-\nabla u,\tau\rangle_{H^{s-\beta}})  + \frac k 2\|\nabla u\|^2_{H^{s-\beta}} + \|\Lambda^{\beta}\tau\|^2_{H^s} \leq 0,
	\end{align}
	where $k$ is a sufficiently small constant.
\end{theo}

Our main results can be stated as follows:
\begin{theo}\label{th2}
Let $\beta=1$. Let $(u,\tau)$ be a strong solution of \eqref{eq0} with the initial data $(u_0,\tau_0)$ under the condition in Theorem \ref{th1}. Suppose that $(u_0,\tau_0)\in \dot{B}^{-1}_{2,\infty}$, then there exists $C>0$ such that for every $t>0$ and $s_1\in[0,s]$, there holds
\begin{align}\label{high1}
\|\Lambda^{s_1}(u,\tau)\|_{L^2} \leq C(1+t)^{-\frac{1+{s_1}}{2}}.
\end{align}
In addition, if $(u_0,\tau_0)\in H^{s+1}$ and $0<c_0=|\int_{\mathbb{R}^2}(u_0,\tau_0)dx|$, then there exists $C_0(s_1,c_0)\leq C$ such that
\begin{align}\label{low1}
\|\Lambda^{s_1}(u,\tau)\|_{L^2}\geq \frac {C_0}{2}(1+t)^{-\frac {s_1+1}{2}}.
\end{align}	
\end{theo}
\begin{theo}\label{th3}
	Let $\beta\in[\frac 1 2, 1)$. Let $(u,\tau)$ be a strong solution of \eqref{eq0} with the initial data $(u_0,\tau_0)$ under the condition in Theorem \ref{th1}. If $(u_0,\tau_0)\in \dot{B}^{-1}_{2,\infty}$, then there exists $C>0$ such that for every $t>0$,
	\begin{align}\label{high2}
	\|\Lambda^{s}(u,\tau)\|_{L^2} \leq C(1+t)^{-\frac{{s}+1}{2\beta}}.
	\end{align}
	Furthermore, suppose that $(u_0,\tau_0)\in H^{s+\beta}$ and $0<|\int_{\mathbb{R}^2}(u_0,\tau_0)dx|$, then there exists $C_\beta\leq C$ such that
	\begin{align}\label{low2}
	\|\Lambda^{s}(u,\tau)\|_{L^2}\geq \frac {C_\beta}{2}(1+t)^{-\frac {s+1}{2\beta}}.
	\end{align}	
\end{theo}

\begin{rema}
The classical result about large time behaviour often supposed that the initial data belongs to $L^1$(See \cite{Schonbek1985}).
Since $L^1\hookrightarrow \dot{B}^{-1}_{2,\infty}$, it follows that the above results still hold true when $(u_0,\tau_0)\in L^1$.
\end{rema}
The paper is organized as follows. In Section 2 we introduce some lemmas which will be used in the sequel. In Section 3 we prove optimal decay rate for the 2-D generalized Oldroyd-B model in critical case by virtue of a different Fourier splitting method and the
time weighted energy estimate. In Section 4 we prove optimal decay rate for the highest derivative of the solution to 2-D generalized Oldroyd-B model in fractional case by the improved Fourier splitting method and the time weighted energy estimate.

\section{Preliminaries}
In this section we introduce some lemmas which will be used in the sequel.

The Littlewood-Paley decomposition theory and and Besov spaces are given as follows.
\begin{lemm}\cite{Bahouri2011}\label{Lemma0}
	Let $\mathcal{C}=\{\xi\in\mathbb{R}^2:\frac 3 4\leq|\xi|\leq\frac 8 3\}$. There exists radial function $\varphi$, valued in the interval $[0,1]$, belonging respectively to $\mathcal{D}(\mathcal{C})$, and
	$$ \forall\xi\in\mathbb{R}^2\backslash\{0\},\ \sum_{j\in\mathbb{Z}}\varphi(2^{-j}\xi)=1, $$
	$$ |j-j'|\geq 2\Rightarrow\mathrm{Supp}\ \varphi(2^{-j}\cdot)\cap \mathrm{Supp}\ \varphi(2^{-j'}\cdot)=\emptyset. $$
	Moreover, there holds
	$$ \forall\xi\in\mathbb{R}^2\backslash\{0\},\ \frac 1 2\leq\sum_{j\in\mathbb{Z}}\varphi^2(2^{-j}\xi)\leq 1. $$
\end{lemm}
$\mathcal{F}$ denotes the Fourier transform and  its inverse is represented by $\mathcal{F}^{-1}$.
Suppose that $u$ is a tempered distribution in $\mathcal{S}'_h(\mathbb{R}^2)$. For all $j\in\mathbb{Z}$, define
$$\dot{\Delta}_j u=\mathcal{F}^{-1}(\varphi(2^{-j}\cdot)\mathcal{F}u).$$
Then the Littlewood-Paley decomposition is defined by:
$$ u=\sum_{j\in\mathbb{Z}}\dot{\Delta}_j u \quad \text{in}\ \mathcal{S}'(\mathbb{R}^2). $$
Let $s\in\mathbb{R},\ 1\leq p,r\leq\infty.$ The homogeneous Besov space $\dot{B}^s_{p,r}$ is given as follows
$$ \dot{B}^s_{p,r}=\{u\in \mathcal{S}'_h:\|u\|_{\dot{B}^s_{p,r}}=\Big\|(2^{js}\|\dot{\Delta}_j u\|_{L^p})_j \Big\|_{l^r(\mathbb{Z})}<\infty\}.$$

Take $\Lambda f=\mathcal{F}^{-1}(|\xi|\mathcal{F}(f)),$
we introduce the Gagliardo-Nirenberg inequality of Sobolev type with $d=2$.
\begin{lemm}\cite{1959On}\label{Lemma1}
For $d=2,~p\in[2,+\infty)$ and $0\leq s,s_1\leq s_2$, there holds
$$\|\Lambda^{s}f\|_{L^{p}}\leq C \|\Lambda^{s_1}f\|^{1-\theta}_{L^{2}}\|\Lambda^{s_2} f\|^{\theta}_{L^{2}},$$
where $0\leq\theta\leq1$ and
$$ s+1-\frac 2 p=s_1 (1-\theta)+\theta s_2.$$
Note that we also require that $0<\theta<1$, $0\leq s_1\leq s$, when $p=\infty$.
\end{lemm}
\begin{lemm}\cite{Moser1966A}\label{Lemma2}
Assume that $s\geq 1$, $p,p_1,p_4\in (1,\infty)$ and $\frac 1 p =\frac 1 {p_1}+\frac 1 {p_2}=\frac 1 {p_3}+\frac 1 {p_4}$, then we obtain
\begin{align*}
\|[\Lambda^s, f]g\|_{L^p}\leq C(\|\Lambda^{s}f\|_{L^{p_1}}\|g\|_{L^{p_2}}+\|\nabla f\|_{L^{p_3}}\|\Lambda^{s-1}g\|_{L^{p_4}}).
\end{align*}
\end{lemm}

\section{Optimal decay rate with $\beta=1$}
In this section, we present optimal decay rate for the generalized Oldroyd-B model \eqref{eq1} in critical case with $\beta=1$. Inspired by \cite{He2009}, we consider the coupling effect between $(u ,\tau)$.  We need to introduce the following energy and dissipation functionals for $(u,\tau)$:
$$E_\theta=\|\Lambda^\theta(u,\tau)\|^2_{H^{s-\theta}}+2k\langle \nabla\Lambda^\theta u,  \Lambda^\theta\tau \rangle_{H^{s-1-\theta}},$$
and
$$D_\theta=\frac k 2\|\nabla \Lambda^\theta u\|^2_{H^{s-1-\theta}}+\|\nabla\Lambda^\theta\tau\|^2_{H^{s-\theta}},$$
where $\theta=0~or~1$. The key point is the refinement of Schonbek's \cite{Schonbek} strategy. Using a different Fourier splitting method and taking Fourier transform in \eqref{eq1}, we obtain the initial $L^2$ decay rate in following proposition.
\begin{prop}\label{prop1}
	Assume that $(u_0,\tau_0)$ satisfy the condition in Theorem \ref{th1} and $(u_0,\tau_0)\in \dot{B}^{-1}_{2,\infty}$.
	For any $l\in N^{+}$, there exists constant $C_l$ such that
	\begin{align}\label{decay1}
	E_0(t)\leq C_{l}\ln^{-l}(e+t).
	\end{align}
\end{prop}
\begin{proof}
	Taking $\beta=1$ in Theorem \ref{th1}, we get the following global energy estimate:
	\begin{align}\label{3ineq1}
	\frac d {dt} E_0(t)+D_0(t)\leq 0.
	\end{align}
	Denote that $S_0(t)=\{\xi:|\xi|^2\leq 2C_2\frac {f'(t)} {f(t)}\}$, where $f(t)=\ln^{3}(e+t)$ and $C_2$ is large enough. By \eqref{3ineq1}, we deduce that
	\begin{align}\label{3ineq2}
	\frac d {dt} [f(t)E_0(t)]+C_2 f'(t)(\frac k 2\|u\|^2_{H^{s}}+\|\tau\|^2_{H^{s}})\leq Cf'(t)\int_{S_0(t)}|\hat{u}(\xi)|^2+|\hat{\tau}(\xi)|^2d\xi.
	\end{align}
	
	The $L^2$ estimate to the low frequency part of $(u,\tau)$ is useful for studying time decay rates. Applying Fourier transform to \eqref{eq1}, we have
	\begin{align}\label{eq2}
	\left\{
	\begin{array}{ll}
	\hat{u}^{j}_t+i\xi_{j} \hat{P}-i\xi_{k} \hat{\tau}^{jk}=\hat{F}^j,  \\[1ex]
	\hat{\tau}^{jk}_t+|\xi|^2\hat{\tau}^{jk}-\frac i 2(\xi_{k} \hat{u}^j+\xi_{j} \hat{u}^k)=\hat{G}^{jk},\\[1ex]
	\end{array}
	\right.
	\end{align}
	where $F=-u\cdot\nabla u$ and $G=-u\cdot\nabla \tau-Q(\nabla u, \tau)$. According to  $\mathcal{R}e[i\xi\otimes\bar{\hat{u}}(t,\xi):\hat{\tau}]+\mathcal{R}e[i\xi\otimes\hat{u}:\bar{\hat{\tau}}]=0$
	and \eqref{eq2}, we infer that
	\begin{align}\label{eq3}
	\frac 1 2 \frac d {dt} (|\hat{u}|^2+|\hat{\tau}|^2)
	+|\xi|^2|\hat{\tau}|^2  =\mathcal{R}e[\hat{F}\cdot\bar{\hat{u}}]+\mathcal{R}e[\hat{G}\bar{\hat{\tau}}].
	\end{align}
	Integrating \eqref{eq3} in time on $[0,t]$, we get
	\begin{align}\label{3ineq3}
	|\hat{u}|^2+|\hat{\tau}|^2
	\leq C(|\hat{u}_0|^2+|\hat{\tau}_0|^2)+C\int_{0}^{t}|\hat{F}\cdot\bar{\hat{u}}|+|\hat{G}\cdot\bar{\hat{\tau}}|ds'.
	\end{align}
	Integrating \eqref{3ineq3} over $S_0(t)$ with $\xi$, then we obtain the following estimate for \eqref{eq2}:
	\begin{align}\label{3ineq4}
	\int_{S_0(t)}|\hat{u}|^2+|\hat{\tau}|^2 d\xi
	\leq C\int_{S_0(t)} |\hat{u}_0|^2+|\hat{\tau}_0|^2 d\xi +C\int_{S_0(t)}\int_{0}^{t}|\hat{F}\cdot\bar{\hat{u}}|+|\hat{G}\cdot\bar{\hat{\tau}}| ds'd\xi.
	\end{align}
	According to $E_0(0)<\infty$, $(u_0,\tau_0)\in \dot{B}^{-1}_{2,\infty}$ and applying Lemma \ref{Lemma0}, we infer that
	\begin{align}\label{3ineq5}
	\int_{S_0(t)}(|\hat{u}_0|^2+|\hat{\tau}_0|^2)d\xi
	&\leq\sum_{j\leq \log_2[\frac {4} {3}C_2^{\frac 1 2 }\sqrt{\frac {f'(t)}{f(t)}}]}\int_{\mathbb{R}^{2}} 2\varphi^2(2^{-j}\xi)(|\hat{u}_0|^2+|\hat{\tau}_0|^2)d\xi  \\ \notag
	&\leq\sum_{j\leq \log_2[\frac {4} {3}C_2^{\frac 1 2 }\sqrt{\frac {f'(t)}{f(t)}}]}C(\|\dot{\Delta}_j u_0\|^2_{L^2}+\|\dot{\Delta}_j \tau_0\|^2_{L^2}) \\ \notag
	&\leq C\frac {f'(t)}{f(t)}\|(u_0,\tau_0)\|^2_{\dot{B}^{-1}_{2,\infty}}.
	\end{align}
	By \eqref{3ineq1} and Minkowski's inequality, we obtain
	\begin{align}\label{3ineq6}
	\int_{S_0(t)}\int_{0}^{t}|\hat{F}\cdot\bar{\hat{u}}|+|\hat{G}\cdot\bar{\hat{\tau}}|  ds'd\xi
	&=\int_{0}^{t}\int_{S_0(t)}|\hat{F}\cdot\bar{\hat{u}}|+|\hat{G}\cdot\bar{\hat{\tau}}|d\xi ds'  \\ \notag
	&\leq C\sqrt{\frac {f'(t)} {f(t)}} \int_{0}^{t}(\|u\|^2_{L^{2}}+\|\tau\|^2_{L^{2}})D_0(s')^{\frac 1 2}ds'  \\ \notag
	&\leq C\sqrt{\frac {f'(t)} {f(t)}}(1+t)^{\frac 1 2}.
	\end{align}
	It follows from \eqref{3ineq4}-\eqref{3ineq6} that
	\begin{align}\label{3ineq7}
	\int_{S_0(t)}|\hat{u}|^2+|\hat{\tau}|^2 d\xi\leq C\ln^{-\frac 1 2}(e+t).
	\end{align}
	Combining \eqref{3ineq2} and \eqref{3ineq7}, we have
	\begin{align}\label{3ineq8}
	\frac d {dt} [f(t)E_0(t)]\leq Cf'(t)\ln^{-\frac 1 2}(e+t).
	\end{align}
	Consequently, we get the initial time decay rate:
	\begin{align}\label{3ineq9}
	E_0(t)\leq C\ln^{-\frac 1 2}(e+t).
	\end{align}
	By virtue of the bootstrap argument, for any $l\in N^{+}$, we infer that $E_0\leq C_{l}\ln^{-l}(e+t).$
\end{proof}	
We now consider the the initial time decay rate of $E_1(t)$.
\begin{prop}\label{prop2}
	Let $(u_0,\tau_0)\in \dot{B}^{-1}_{2,\infty}$. Under the condition in Theorem \ref{th1},
	for any $l\in N^{+}$, then there exists a constant $C_l$ such that
	\begin{align}\label{decay2}
	E_1(t)\leq C_{l}(1+t)^{-1}\ln^{-l}(e+t).
	\end{align}
\end{prop}
\begin{proof}	
	Since $\langle u\cdot\nabla u,\Delta u\rangle=0$ with $d=2$,  then we have
	\begin{align}\label{3ineq10}
	\frac 1 2 \frac d {dt} \|\nabla(u,\tau)\|^2_{L^2}+ \|\nabla^2\tau\|^2_{L^2}=\langle u\cdot\nabla \tau+Q(u,\tau),\Delta \tau\rangle\leq C\delta D_1 ,
	\end{align}
	and
	\begin{align}\label{3ineq11}
	\frac d {dt} \langle \Delta\tau, \nabla u\rangle+ \frac 1 2\|\nabla^2 u\|^2_{L^2}
	&=\langle div~(u\cdot\nabla \tau+Q(u,\tau)-\Delta \tau),\Delta u\rangle \\ \notag
	&-\langle \nabla\mathbb{P}(u\cdot\nabla u-div~\tau),\Delta \tau\rangle  \\ \notag
	&\leq C\delta D_1+C\|\nabla^2 \tau\|^2_{L^2}.
	\end{align}
	Using Lemma \ref{Lemma2}, we obtain
	\begin{align}\label{3ineq12}
	\frac 1 2\frac d {dt} \|\Lambda^s(u,\tau)\|^2_{L^2}+ \|\nabla\Lambda^s\tau\|^2_{L^2}&=-\langle [\Lambda^s,u\cdot\nabla] u,\Lambda^s u\rangle-\langle \Lambda^s(u\cdot\nabla \tau+Q(u,\tau)),\Lambda^s \tau\rangle \\ \notag
	&\leq C(\|\nabla u\|_{L^\infty}\|\Lambda^s u\|^2_{L^2}+\|\tau\|_{L^\infty}\|\Lambda^s u\|_{L^2}\|\Lambda^{s+1} \tau\|_{L^2}  \\ \notag
	&+\|u\|_{L^\infty}\|\Lambda^s \tau\|_{L^2}\|\Lambda^{s+1} \tau\|_{L^2}+\|\nabla u\|_{L^\infty}\|\Lambda^{s-1} \tau\|_{L^2}\|\Lambda^{s+1} \tau\|_{L^2}).
	\end{align}
	By Lemma \ref{Lemma1}, we have $\|\nabla u\|_{L^\infty}\|\Lambda^{s-1} \tau\|_{L^2}\leq C\| u\|^{\frac {s-2} s}_{L^2}\|\nabla^s u\|^{\frac {2} s}_{L^2}\|\tau\|^{\frac {1} s}_{L^2}\|\Lambda^{s} \tau\|^{\frac {s-1} s}_{L^2}$. This together with \eqref{3ineq12} and Theorem \ref{th1} ensure that
	\begin{align}\label{3ineq13}
	\frac d {dt} \|\Lambda^s(u,\tau)\|^2_{L^2}+ \|\nabla\Lambda^s\tau\|^2_{L^2}\leq C\delta D_1.
	\end{align}
	Using Lemmas \ref{Lemma1} and \ref{Lemma2}, we deduce that
	\begin{align}\label{3ineq14}
			&\frac d {dt} \langle \Lambda^{s-1}\tau, -\nabla\Lambda^{s-1} u\rangle+ \frac 1 2\|\nabla\Lambda^{s-1} u\|^2_{L^2}  \\ \notag
			&=\langle \Lambda^{s-1}(u\cdot\nabla \tau+Q(u,\tau)-\Delta \tau),\nabla\Lambda^{s-1} u\rangle \\ \notag
			&-\langle \Lambda^{s-1}\mathbb{P}(u\cdot\nabla u-div~ \tau),div~\Lambda^{s-1} \tau\rangle  \\ \notag
			&\leq C\|\Lambda^s u\|_{L^2}(\|u\|_{L^\infty}\|\Lambda^s \tau\|_{L^2}+\|\tau\|_{L^\infty}\|\Lambda^s u\|_{L^2}+\|\Lambda^{s+1} \tau\|_{L^2})  \\ \notag
			&+ C\|\Lambda^s \tau\|_{L^2}(\|u\|_{L^\infty}\|\Lambda^s u\|_{L^2}+\|\Lambda^{s} \tau\|_{L^2})	\\ \notag
			&\leq C\delta D_1+C\|\Lambda^s \tau\|^2_{H^1}.
		\end{align}	
	Combining \eqref{3ineq10}-\eqref{3ineq14}, we obtain
	\begin{align}\label{3ineq15}
	\frac d {dt} E_1+D_1\leq 0,
	\end{align}
	which implies that
	\begin{align}\label{3ineq16}
	\frac d {dt} [f(t)E_1]+C_2 f'(t)(\frac k 2\|\nabla u\|^2_{H^{s-1}}+\|\nabla \tau\|^2_{H^{s-1}})
	\leq Cf'(t)\int_{S_0(t)}|\xi|^2(|\hat{u}(\xi)|^2+|\hat{\tau}(\xi)|^2) d\xi.
	\end{align}
	By Proposition \ref{prop1}, we obtain
	\begin{align}\label{3ineq17}
	f'(t)\int_{S_0(t)}|\xi|^2(|\hat{u}(\xi)|^2+|\hat{\tau}(\xi)|^2) d\xi\leq C (1+t)^{-2}\ln^{-l+1}(e+t).
	\end{align}
	According to \eqref{3ineq16}-\eqref{3ineq17}, we infer that
	\begin{align}\label{3ineq18}
	E_1\leq C (1+t)^{-1}\ln^{-l}(e+t).
	\end{align}
	Therefore, we complete the proof of Proposition \ref{prop2}.
\end{proof}

By virtue of the standard method, we can't immediately obtain the optimal decay rate. However, we obtain a weak result as follows.
\begin{prop}\label{prop3}
	Assume that $(u_0,\tau_0)$ satisfy the condition in Proposition \ref{prop1}, there holds
	\begin{align}\label{decay3}
	E_0(t)\leq C(1+t)^{-\frac 1 2},
	\end{align}
	and
	\begin{align}\label{decay4}
	E_1(t)\leq C(1+t)^{-\frac 3 2}.
	\end{align}
\end{prop}
\begin{proof}
	Define $S(t)=\{\xi:|\xi|^2\leq C_2(1+t)^{-1}\}$ with $C_2$ large enough.
	By \eqref{3ineq1}, we infer that
	\begin{align}\label{3ineq19}
	\frac d {dt} E_0(t)+\frac {kC_2} {2(1+t)}\| u\|^2_{H^s}+\frac {C_2} {1+t}\|\tau\|^2_{H^{s}}
	\leq \frac {C} {1+t}\int_{S(t)}|\hat{u}(\xi)|^2+|\hat{\tau}(\xi)|^2 d\xi.
	\end{align}
	Integrating \eqref{3ineq3} over $S(t)$ with $\xi$, then we get
	\begin{align}\label{3ineq20}
	\int_{S(t)}|\hat{u}|^2+|\hat{\tau}|^2d\xi
	\leq C\int_{S(t)} |\hat{u}_0|^2+|\hat{\tau}_0|^2d\xi  +C\int_{S(t)}\int_{0}^{t}|\hat{F}\cdot\bar{\hat{u}}|+|\hat{G}\cdot\bar{\hat{\tau}}|ds'd\xi.
	\end{align}
	According to $(u_0,\tau_0)\in H^s\cap\dot{B}^{-1}_{2,\infty}$ and applying Lemma \ref{Lemma0}, we deduce that
	\begin{align}
	\int_{S(t)}|\hat{u}_0|^2+|\hat{\tau}_0|^2 d\xi
	&\leq\sum_{j\leq \log_2[\frac {4} {3}C_2^{\frac 1 2 }(1+t)^{-\frac 1 2}]}\int_{\mathbb{R}^{2}} 2\varphi^2(2^{-j}\xi)(|\hat{u}_0|^2+|\hat{\tau}_0|^2)d\xi \\ \notag
	&\leq\sum_{j\leq \log_2[\frac {4} {3}C_2^{\frac 1 2 }(1+t)^{-\frac 1 2}]}2(\|\dot{\Delta}_j u_0\|^2_{L^2}+\|\dot{\Delta}_j \tau_0\|^2_{L^2}) \\ \notag
	&\leq C(1+t)^{-1}\|(u_0,\tau_0)\|^2_{\dot{B}^{-1}_{2,\infty}},
	\end{align}
    and
	\begin{align}\label{3ineq21}
	\int_{S(t)}\int_{0}^{t}|\hat{F}\cdot\bar{\hat{u}}|+|\hat{G}\cdot\bar{\hat{\tau}}|ds'd\xi\leq C(1+t)^{-\frac 1 2} \int_{0}^{t}(\|u\|^2_{L^{2}}+\|\tau\|^2_{L^{2}})(\|\nabla u\|_{L^{2}}+\|\nabla \tau\|_{L^{2}})ds'.
	\end{align}
	Combining \eqref{3ineq19}-\eqref{3ineq21}, we have
	\begin{align}\label{3ineq22}
	&\frac d {dt} E_0(t)+\frac {kC_2} {2(1+t)}\| u\|^2_{H^s}+\frac {C_2} {1+t}\|\tau\|^2_{H^{s}} \\ \notag
	&\leq \frac {C} {1+t}[(1+t)^{-1}+(1+t)^{-\frac 1 2} \int_{0}^{t}(\|u\|^2_{L^{2}}+\|\tau\|^2_{L^{2}})(\|\nabla u\|_{L^{2}}+\|\nabla \tau\|_{L^{2}})ds'],
	\end{align}
	which implies that
	\begin{align}\label{3ineq23}
	(1+t)^{\frac 3 2}E_0(t)\leq C(1+t)^{\frac 1 2}
	+C(1+t)\int_{0}^{t}(\|u\|^2_{L^{2}}+\|\tau\|^2_{L^{2}})(\|\nabla u\|_{L^{2}}+\|\nabla \tau\|_{L^{2}})ds'.
	\end{align}
	Let $N(t)=\sup_{0\leq s\leq t}(1+s)^{\frac 1 2}E_0(s)$. By \eqref{3ineq23}, we obtain
	\begin{align}\label{3ineq24}
	N(t)\leq C+C\int_{0}^{t}N(s)(1+s)^{-\frac 1 2}(\|\nabla u\|_{L^{2}}+\|\nabla\tau\|_{L^{2}})ds'.
	\end{align}
	Applying Gronwall's inequality and Proposition \ref{prop2}, we deduce that $N(t)\leq C$,
	which implies that
	\begin{align}\label{3ineq25}
	E_0\leq C(1+t)^{-\frac 1 2}.
	\end{align}
	According to \eqref{3ineq15}, we infer that
	\begin{align}\label{3ineq26}
	&\frac d {dt} E_1+\frac { C_2} {1+t}(\frac k 2\|\nabla u\|^2_{H^{s-1}}+\|\nabla \tau\|^2_{H^{s-1}})
	\leq \frac {C} {1+t}\int_{S(t)}|\xi|^2(|\hat{u}(\xi)|^2+|\hat{\tau}(\xi)|^2) d\xi.
	\end{align}
	By \eqref{3ineq25}, we get
	\begin{align}\label{3ineq27}
	\frac {C} {1+t}\int_{S(t)}|\xi|^2(|\hat{u}(\xi)|^2+|\hat{\tau}(\xi)|^2) d\xi\leq C (1+t)^{-2}(\|u\|^2_{L^2}+\|\tau\|^2_{L^2})\leq C (1+t)^{-\frac 5 2}.
	\end{align}
	This together with \eqref{3ineq1}, \eqref{3ineq25} and \eqref{3ineq26} ensure that
	\begin{align}\label{3ineq28}
	E_1\leq C(1+t)^{-\frac 3 2}.
	\end{align}
	We thus complete the proof of Proposition \ref{prop3}.
\end{proof}

Using Proposition \ref{prop3}, we can prove that the solution of \eqref{eq1} belongs to
some Besov space with negative index.
\begin{lemm}\label{Lemma3}
	Let $0<\alpha,\sigma\leq 1$ and $\sigma<2\alpha$. Under the condition in Proposition \ref{prop1}. If
	\begin{align}\label{de0}
	E_0(t)\leq C(1+t)^{-\alpha},~~~~E_1(t)\leq C(1+t)^{-\alpha-1},
	\end{align}
	then we have
	\begin{align}\label{3ineq29}
	(u,\tau)\in L^{\infty}(0,\infty;\dot{B}^{-\sigma}_{2,\infty}).
	\end{align}
\end{lemm}
\begin{proof}
	Applying $\dot{\Delta}_j$ to \eqref{eq1}, we have
	\begin{align}\label{eq4}
	\left\{
	\begin{array}{ll}
	\dot{\Delta}_j u_t+\nabla\dot{\Delta}_j P-div~\dot{\Delta}_j\tau=\dot{\Delta}_j F,  \\[1ex]
	\dot{\Delta}_j \tau_t-\Delta\dot{\Delta}_j \tau-\dot{\Delta}_j D(u)=\dot{\Delta}_j G. \\[1ex]
	\end{array}
	\right.
	\end{align}
	We first deduce from \eqref{eq4} that
	\begin{align}\label{3ineq30}
	\frac d {dt}(\|\dot{\Delta}_j u\|^2_{L^2}+\|\dot{\Delta}_j \tau\|^2_{L^2})+2\|\nabla\dot{\Delta}_j \tau\|^2_{L^2}
	\leq C(\|\dot{\Delta}_j F\|_{L^2}\|\dot{\Delta}_j u\|_{L^2}+\|\dot{\Delta}_j G\|_{L^2}\|\dot{\Delta}_j \tau\|_{L^2}).
	\end{align}
	Applying $2^{-2j\sigma}$ to \eqref{3ineq30} and taking $l^\infty$ norm, we obtain
	\begin{align}\label{3ineq31}
	\frac d {dt}(\|u\|^2_{\dot{B}^{-\sigma}_{2,\infty}}+\|\tau\|^2_{\dot{B}^{-\sigma}_{2,\infty}}) \leq C(\|F\|_{\dot{B}^{-\sigma}_{2,\infty}}\|u\|_{\dot{B}^{-\sigma}_{2,\infty}}
	+\|G\|_{\dot{B}^{-\sigma}_{2,\infty}}\|\tau\|_{\dot{B}^{-\sigma}_{2,\infty}}).
	\end{align}
	Let $M(t)=\sum_{0\leq s\leq t} \|u\|_{\dot{B}^{-\sigma}_{2,\infty}}+\|\tau\|_{\dot{B}^{-\sigma}_{2,\infty}}$. According to \eqref{3ineq31}, we infer that
	\begin{align}\label{3ineq32}
	M^2(t)&\leq CM^2(0)+CM(t)\int_0^{t}\|F\|_{\dot{B}^{-\sigma}_{2,\infty}}+\|G\|_{\dot{B}^{-\sigma}_{2,\infty}}ds'.
	\end{align}
	Using \eqref{de0} and inclusion between Lesbesgue and Besov space, we deduce that
	\begin{align}\label{3ineq33}
	\int_0^{t}\|(F,G)\|_{\dot{B}^{-\sigma}_{2,\infty}}ds'
	&\leq C\int_0^{t}\|(F,G)\|_{L^{\frac 2 {\sigma+1}}}ds'  \\ \notag
	&\leq C\int_0^{t}(\|u\|_{L^{\frac 2 {\sigma}}}+\|\tau\|_{L^{\frac 2 {\sigma}}})(\|\nabla u\|_{L^2}+\|\nabla\tau\|_{L^2})ds'   \\ \notag
	&\leq C\int_0^{t}\|(u,\tau)\|^\sigma_{L^{2}}\|\nabla(u,\tau)\|^{2-\sigma}_{L^{2}}ds'   \\ \notag
	&\leq C\int_0^{t}(1+s)^{-(1+\alpha-\frac \sigma 2)}ds'\leq C.
	\end{align}
	Combining \eqref{3ineq32} and \eqref{3ineq33}, we get $M(t)\leq C$.
\end{proof}

\begin{lemm}\label{Lemma4}
	Let $0<\beta,\sigma\leq 1$ and $\frac {1} {2} \leq\alpha$. Under the condition in Proposition \ref{prop1}.
	For any $t\in [0,+\infty)$, if
	\begin{align}\label{de1}
	E_0(t)\leq C(1+t)^{-\alpha},~~~~E_1(t)\leq C(1+t)^{-\alpha-1},
	\end{align}
	and
	\begin{align}\label{3ineq34}
	(u,\tau)\in L^{\infty}(0,\infty;\dot{B}^{-\sigma}_{2,\infty}),
	\end{align}
	then there exists a constant $C$ such that
	\begin{align}\label{de2}
	E_0(t)\leq C(1+t)^{-\beta}~~~~and~~~~E_1(t)\leq C(1+t)^{-\beta-1},
	\end{align}
	where $\beta<\frac {\sigma+1} {2}$ for $\alpha=\frac {1} {2}$ and $\beta=\frac {\sigma+1} {2}$ for $\alpha>\frac {1} {2}$.
\end{lemm}
\begin{proof}
	According to the proof of Proposition \ref{prop2}, we have
	\begin{align}\label{3ineq35}
	&\frac d {dt} E_0(t)+\frac {kC_2} {2(1+t)}\| u\|^2_{H^s}+\frac {C_2} {1+t}\|\tau\|^2_{H^{s}}  \\ \notag
	&\leq \frac {CC_2} {1+t}((1+t)^{-1}+\int_{S(t)}\int_{0}^{t}|\hat{F}\cdot\bar{\hat{u}}|+|\hat{G}\cdot\bar{\hat{\tau}}|dsd\xi).
	\end{align}
	Using \eqref{de1} and \eqref{3ineq34}, we infer that
	\begin{align}\label{3ineq36}
	\int_{S(t)}\int_{0}^{t}|\hat{F}\cdot\bar{\hat{u}}|+|\hat{G}\cdot\bar{\hat{\tau}}|dsd\xi
	&\leq C\int_{0}^{t}(\|F\|_{L^{1}}\int_{S(t)}|\hat{u}|d\xi+\|G\|_{L^{1}}\int_{S(t)}|\hat{\tau}|d\xi)ds \\ \notag
	&\leq C(1+t)^{-\frac 1 2}\int_{0}^{t}(\|F\|_{L^{1}}+\|G\|_{L^{1}})(\int_{S(t)}|\hat{u}|^2+|\hat{\tau}|^2d\xi)^{\frac 1 2} ds  \\ \notag
	&\leq C(1+t)^{-\frac {\sigma+1} {2}}M(t)\int_{0}^{t}\|F\|_{L^{1}}+\|G\|_{L^{1}}ds  \\ \notag
	&\leq C(1+t)^{-\frac {\sigma+1} {2}}\int_{0}^{t}(1+s)^{-(\alpha+\frac 1 2)}ds  \\ \notag
	&\leq C(1+t)^{-\beta}.
	\end{align}
	By virtue of \eqref{3ineq35} and \eqref{3ineq36}, we get
	\begin{align}\label{3ineq37}
	E_0(t)\leq C(1+t)^{-\beta}.
	\end{align}
	Moreover, we have
	\begin{align}
	\frac d {dt} E_1+\frac { C_2} {1+t}(\frac k 2\|\nabla u\|^2_{H^{s-1}}+\|\nabla \tau\|^2_{H^{s-1}}) &\leq \frac {C} {1+t}\int_{S(t)}|\xi|^2(|\hat{u}(\xi)|^2+|\hat{\tau}(\xi)|^2) d\xi.  \\ \notag
    &\leq C (1+t)^{-2}(\|u\|^2_{L^2}+\|\tau\|^2_{L^2})   \\ \notag
    &\leq C (1+t)^{-2-\beta},
	\end{align}
	which implies that $E_1\leq C(1+t)^{-1-\beta}$.
\end{proof}
We improve the decay rates in $H^1$ by Lemmas \ref{Lemma3} and \ref{Lemma4}. However, considering the decay rate for the highest derivative of the solution to \eqref{eq1}, the main difficulty is unclosed energy estimate. To overcome this difficulty, we introduce a new method which flexibly combines the Fourier splitting method and the time weighted energy estimate.
\begin{prop}\label{prop4}
	Assume that $(u_0,\tau_0)$ satisfy the condition in Proposition \ref{prop1}, then there exists a constant $C$ such that
	\begin{align}\label{decay5}
	\|\Lambda^{s_1}(u,\tau)\|_{L^2} \leq C(1+t)^{-\frac{1+{s_1}}{2}},
	\end{align}
	where $s_1\in[0,s]$.
\end{prop}
\begin{proof}
We first improve the decay rate in Proposition \ref{prop3}. According to Proposition \ref{prop3} and Lemma \ref{Lemma3} with $\sigma=\alpha=\frac 1 2$, we have
\begin{align*}
(u,\tau)\in L^{\infty}(0,\infty;\dot{B}^{-\frac 1 2}_{2,\infty}).
\end{align*}
By virtue of Lemma \ref{Lemma4} with $\alpha=\sigma=\frac 1 2$ and $\beta=\frac 5 8$, we infer that
\begin{align*}
E_0(t)\leq C(1+t)^{-\frac 5 8}~~~~and~~~~E_1(t)\leq C(1+t)^{-\frac 5 8-1}.
\end{align*}
Taking $\sigma=1$ and $\alpha=\frac 5 8$ in Lemma \ref{Lemma3}, we get
\begin{align}
(u,\tau)\in L^{\infty}(0,\infty;\dot{B}^{-1}_{2,\infty}).
\end{align}
Taking advantage of Lemma \ref{Lemma4} again with $\alpha=\frac 5 8$ and $\sigma=\beta=1$, we deduce that
\begin{align}\label{lineq1}
E_0(t)\leq C(1+t)^{-1}~~~~and~~~~E_1(t)\leq C(1+t)^{-2}.
\end{align}

Then, we introduce some new energy and dissipation functionals for $(u,\tau)$ as follows:
$$\widetilde{E}_s=k(1+t)\|\Lambda^s(u,\tau)\|^2_{L^2}+\langle \Lambda^{s-1}\tau, -\nabla\Lambda^{s-1} u\rangle,$$
and
$$\widetilde{D}_s=k(1+t)\|\nabla\Lambda^s\tau\|^2_{L^2}+ \frac 1 4\|\nabla\Lambda^{s-1} u\|^2_{L^2},$$
where $k$ is small enough.
Using \eqref{3ineq12}, \eqref{3ineq14}, \eqref{lineq1} and Lemmas \ref{Lemma1}-\ref{Lemma2}, we deduce that
	\begin{align}\label{hineq1}
	\frac d {dt} \widetilde{E}_s+ 2\widetilde{D}_s
	&\leq k\|\Lambda^s(u,\tau)\|^2_{L^2}+Ck(1+t)(\|\nabla u\|_{L^\infty}\|\Lambda^s u\|^2_{L^2}+\|\tau\|_{L^\infty}\|\Lambda^s u\|_{L^2}\|\Lambda^{s+1} \tau\|_{L^2}  \\ \notag
	&+\|u\|_{L^\infty}\|\Lambda^s \tau\|_{L^2}\|\Lambda^{s+1} \tau\|_{L^2}+\|\nabla u\|_{L^\infty}\|\Lambda^{s-1} \tau\|_{L^2}\|\Lambda^{s+1} \tau\|_{L^2}), \\ \notag
	&+ C\|\Lambda^s u\|_{L^2}(\|u\|_{L^\infty}\|\Lambda^s \tau\|_{L^2}+\|\tau\|_{L^\infty}\|\Lambda^s u\|_{L^2}+\|\Lambda^{s+1} \tau\|_{L^2})  \\ \notag
	&+ C\|\Lambda^s \tau\|_{L^2}(\|u\|_{L^\infty}\|\Lambda^s u\|_{L^2}+\|\Lambda^{s} \tau\|_{L^2})	\\ \notag
	&\leq (Ck+\frac 1 2)\widetilde{D}_s+Ck(1+t)\|u\|^2_{L^\infty}\|\Lambda^s \tau\|^2_{L^2}+C\|\Lambda^s \tau\|^2_{L^2}   \\ \notag
	&+Ck(1+t)\|\nabla u\|_{L^\infty}\|\Lambda^{s-1} \tau\|_{L^2}\|\Lambda^{s+1} \tau\|_{L^2}  \\ \notag
	&\leq (Ck+\frac 1 2)\widetilde{D}_s+C\|\Lambda^{s-1} \tau\|_{L^2}\|\Lambda^{s+1} \tau\|_{L^2},
	\end{align}
	where we use $\|\Lambda^s \tau\|^2_{L^2}\leq C\|\Lambda^{s-1} \tau\|_{L^2}\|\Lambda^{s+1} \tau\|_{L^2}$. By \eqref{lineq1} and Lemma \ref{Lemma1}, we have
	\begin{align}\label{hineq2}
	C\|\Lambda^{s-1} \tau\|_{L^2}\|\Lambda^{s+1} \tau\|_{L^2}&\leq C\| \tau\|^{\frac 2 {s+1}}_{L^2}\|\Lambda^{s+1} \tau\|^{\frac {2s}{s+1}}_{L^2}  \\ \notag
	&\leq\frac k 4 (1+t)\|\Lambda^{s+1} \tau\|^{2}_{L^2}+C(1+t)^{-s}\| \tau\|^{2}_{L^2}  \\ \notag
	&\leq\frac 1 4 \widetilde{D}_s+C(1+t)^{-s-1}.
	\end{align}
	Combining \eqref{hineq1} and \eqref{hineq2}, we infer that
	\begin{align}\label{hineq3}
	\frac d {dt} \widetilde{E}_s+ \widetilde{D}_s\leq C(1+t)^{-s-1},
	\end{align}
	which implies that	
	\begin{align}\label{hineq4}
	\frac d {dt} \widetilde{E}_s+ kC_2\|\Lambda^s\tau\|^2_{L^2}+ \frac 1 4\|\nabla\Lambda^{s-1} u\|^2_{L^2}&\leq C(1+t)^{-s-1}+kC_2
	\int_{S(t)}|\xi|^{2s}|\hat{\tau}(\xi)|^2 d\xi \\ \notag
	&\leq C(1+t)^{-s-1}.
	\end{align}	
	According to \eqref{hineq4}, we infer that
	\begin{align}\label{hineq5}
	&(1+t)^{s+1} \widetilde{E}_s+\int_{0}^{t}k(1+s')^{s+1} \|\Lambda^s(\tau, u)\|^2_{L^2}ds'  \\  \notag
	&\leq C(1+t)+C\int_{0}^{t}(1+s')^{s}\langle \Lambda^{s-1}\tau, -\nabla\Lambda^{s-1} u\rangle ds' \\ \notag
	&\leq C(1+t)+\frac k 2\int_{0}^{t}(1+s')^{s+1} \|\Lambda^{s} u\|^2_{L^2}ds'+C\int_{0}^{t}(1+s')^{s-1} \|\Lambda^{s-1} \tau\|^2_{L^2}ds'  \\ \notag
	&\leq C(1+t)+\frac k 2\int_{0}^{t}(1+s')^{s+1} \|\Lambda^s(\tau, u)\|^2_{L^2}ds'+C\int_{0}^{t}(1+s')^{1-s} \|\tau\|^2_{L^2}ds'  \\ \notag
	&\leq C(1+t)+\frac k 2\int_{0}^{t}(1+s')^{s+1} \|\Lambda^s(\tau, u)\|^2_{L^2}ds'.
	\end{align}	
	Using Lemma \ref{Lemma1} again, we obtain
	\begin{align}\label{hineq6}
	(1+t)^{s+2}\|\Lambda^s(u,\tau)\|^2_{L^2}&\leq C(1+t)+C(1+t)^{s+1}\langle \Lambda^{s-1}\tau, -\nabla\Lambda^{s-1} u\rangle \\  \notag
	&\leq C(1+t)+k(1+t)^{s+2}\|\Lambda^s u\|^2_{L^2}+C(1+t)^{s}\|\Lambda^{s-1}\tau\|^2_{L^2} \\  \notag
	&\leq C(1+t)+k(1+t)^{s+2}\|\Lambda^s(u,\tau)\|^2_{L^2}+C(1+t)^{2-s}\|\tau\|^2_{L^2},
	\end{align}
	which implies that
	\begin{align}\label{hineq7}
	\|\Lambda^s(u,\tau)\|^2_{L^2}\leq C(1+t)^{-s-1}.
	\end{align}	
We thus complete the proof of Proposition \ref{prop4}.
\end{proof}

{\bf The proof of Theorem \ref{th2}:}
By Proposition \ref{prop4}, we only need to prove the lower bound of the decay rate. We first consider the linear system of \eqref{eq1} with $\beta=1$:
\begin{align}\label{eq5}
\left\{\begin{array}{l}
\partial_tu_{L}+\nabla P_{L}-div~\tau_{L}=0,~~~~div~u_{L}=0,\\[1ex]
\partial_t\tau_{L}-D(u_{L})-\Delta\tau_{L}=0,\\[1ex]
u_{L}|_{t=0}=u_0,~~\tau_{L}|_{t=0}=\tau_0. \\[1ex]
\end{array}\right.
\end{align}
According to Proposition \ref{prop4} and Lemma \ref{Lemma3}, one can deduce that $\|\Lambda^{s_1}(u_L,\tau_L)\|^2_{L^2}\leq C(1+t)^{-s_1-1}$ and $(u_L,\tau_L)\in L^{\infty}(0,\infty;\dot{B}^{-1}_{2,\infty})$. Applying Fourier transform to \eqref{eq5}, we get
\begin{align}\label{eq6}
\left\{
\begin{array}{ll}
\partial_t\hat{u}_{L}^{j}+i\xi_{j} \hat{P}_{L}-i\xi_{k} \hat{\tau}_{L}^{jk}=0,  \\[1ex]
\partial_t\hat{\tau}_{L}^{jk}+|\xi|^2\hat{\tau}_{L}^{jk}-\frac i 2(\xi_{k} \hat{u}_{L}^j+\xi_{j} \hat{u}_{L}^k)=0.\\[1ex]
\end{array}
\right.
\end{align}
We introduce a new weighted energy estimate instead of complex spectral analysis to prove the lower bound of the decay rate. From \eqref{eq6}, we have
\begin{align}\label{hineq8}
\frac 1 2 \frac d {dt} [e^{2|\xi|^2 t}|(\hat{u}_{L},\hat{\tau}_{L})|^2]- |\xi|^2 e^{2|\xi|^2 t}|\hat{u}_{L}|^2=0,
\end{align}
which implies that
\begin{align}\label{hineq9}
|\xi|^{2s_1}|(\hat{u}_{L},\hat{\tau}_{L})|^2=|\xi|^{2s_1}e^{-2|\xi|^2 t}|(\hat{u}_{0},\hat{\tau}_{0})|^2+\int_{0}^{t}2|\xi|^{2s_1+2} e^{-2|\xi|^2 (t-s')}|\hat{u}_{L}|^2ds'.
\end{align}
According to $0<c_0=|\int_{\mathbb{R}^2}(u_0,\tau_0)dx|=|(\hat{u}_{0}(0),\hat{\tau}_{0}(0))|$, we deduce that there exists $\eta>0$ such that $|(\hat{u}_{0}(\xi),\hat{\tau}_{0}(\xi))|\geq \frac {c_0} 2$ if $\xi \in B(0,\eta)$. From \eqref{hineq9}, we have
\begin{align}\label{hineq10}
\|(u_{L},\tau_{L})\|_{\dot{H}^{s_1}}^2&\geq \int_{|\xi|\leq \eta} |\xi|^{2s_1}e^{-2|\xi|^2 t}|(\hat{u}_{0},\hat{\tau}_{0})|^2 d\xi  \\ \notag
&\geq \frac {c^2_0}{4}\int_{|\xi|\leq \eta} |\xi|^{2s_1}e^{-2|\xi|^2 t} d\xi \\ \notag
&\geq C^2_0(1+t)^{-1-s_1},
\end{align}
where $C^2_0=\frac {c^2_0}{4}\int_{|y|\leq \eta} |y|^{2s_1}e^{-2|y|^2} dy$.
Taking $u_{N}=u-u_{L}$, $\tau_{N}=\tau-\tau_{L}$ and $P_{N}=P-P_{L}$, then we immediately obtain $\|\Lambda^{s_1}(u_N,\tau_N)\|^2_{L^2}\leq C(1+t)^{-s_1-1}$ and $(u_N,\tau_N)\in L^{\infty}(0,\infty;\dot{B}^{-1}_{2,\infty})$. Moreover, we have
\begin{align}\label{eq7}
\left\{\begin{array}{l}
\partial_tu_{N}+\nabla P_{N}-div~\tau_{N}=F,~~~~div~u_{N}=0,\\[1ex]
\partial_t\tau_{N}-D(u_{N})-\Delta\tau_{N}=G,\\[1ex]
u_{N}|_{t=0}=\tau_{N}|_{t=0}=0. \\[1ex]
\end{array}\right.
\end{align}
According to \eqref{eq7} and time decay rates for $(u_N,\tau_N)$ and $(u,\tau)$, we deduce that
\begin{align}\label{hineq11}
	\frac 1 2 \frac d {dt} \|(u_N,\tau_N)\|^2_{L^2}+ \|\nabla\tau_N\|^2_{L^2}&=\langle F,u_N \rangle+\langle G,\tau_N \rangle  \\ \notag
	&\leq C\|\nabla u\|_{L^2}\|u\|_{L^4}\|u_N\|_{L^4}+C\|\tau_N\|_{L^4}\|\nabla(\tau,u)\|_{L^2}\|(\tau,u)\|_{L^4} \\ \notag
	&\leq C\delta(1+t)^{-2},
\end{align}
and
\begin{align}\label{hineq12}
	\frac d {dt} \langle \tau_N, -\nabla u_N\rangle+ \frac 1 2\|\nabla u_N\|^2_{L^2}
	&=\langle G+\Delta \tau_N,-\nabla u_N\rangle-\langle \mathbb{P}(F+div ~\tau_N),div~\tau_N\rangle  \\ \notag
	&\leq C(1+t)^{-2}.
\end{align}
Using \eqref{eq7} again, we obtain
\begin{align}\label{hineq13}
\frac 1 2\frac d {dt} \|\nabla u_N\|^2_{L^2}&=\langle F,-\Delta u_N\rangle-\langle div~\tau_N,-\Delta u_N\rangle \\ \notag
&\leq C\delta(1+t)^{-2}.
\end{align}
By \eqref{hineq11}-\eqref{hineq13}, we have
\begin{align}\label{hineq14}
&\frac d {dt} (\|(u_N,\tau_N)\|^2_{L^2}+\|\nabla u_N\|^2_{L^2}  +2k_0\langle \tau_N, -\nabla u_N\rangle)  \\ \notag &+2\|\nabla\tau_N\|^2_{L^2}+k_0\|\nabla u_N\|^2_{L^2}  \\ \notag
&\leq C(\delta+k_0)(1+t)^{-2},
\end{align}
which implies that
	\begin{align}\label{hineq15}
		&\frac d {dt} (\|(u_N,\tau_N)\|^2_{L^2}+\|\nabla u_N\|^2_{L^2}+2k_0\langle \tau_N, -\nabla u_N\rangle)+\frac {k_0 C_1} {2(1+t)}\|u_N\|^2_{H^1}+\frac {C_1} {1+t}\|\tau_N\|^2_{L^2}  \\ \notag
		&\leq \frac {CC_1} {1+t}\int_{S(t)}|\hat{u}_N(\xi)|^2+|\hat{\tau}_N(\xi)|^2 d\xi+C(\delta+k_0)(1+t)^{-2}.
	\end{align}
Similar to \eqref{3ineq20} and \eqref{3ineq36}, we infer that
\begin{align}\label{hineq16}
	\int_{S(t)}|\hat{u}_N|^2+|\hat{\tau}_N|^2d\xi
	&\leq C\int_{S(t)}\int_{0}^{t}|\hat{F}\cdot\bar{\hat{u}}_N|
	+|\hat{G}\cdot\bar{\hat{\tau}}_N|ds'd\xi  \\ \notag
	&\leq C(1+t)^{-1} \int_{0}^{t}\|(u,\tau)\|_{L^{2}}\|\nabla (u,\tau)\|_{L^{2}}\|(u_N,\tau_N)\|_{\dot{B}^{-1}_{2,\infty}}ds'  \\ \notag
	&\leq C\delta(1+t)^{-1}.
\end{align}
According to \eqref{hineq15} and \eqref{hineq16}, we obtain
\begin{align*}
	\|(u_N,\tau_N)\|^2_{L^2}+\|\nabla u_N\|^2_{L^2}\leq C(\delta C_1+k_0)(1+t)^{-1}.
\end{align*}
Applying $\Lambda^{s_1}$ to \eqref{eq7}, $0\leq s_1 \leq s$, we get
\begin{align}\label{eq8}
\left\{\begin{array}{l}
\partial_t\Lambda^{s_1}u_{N}+\nabla\Lambda^{s_1} P_{N}-div~\Lambda^{s_1}\tau_{N}=\Lambda^{s_1}F,\\[1ex]
\partial_t\Lambda^{s_1}\tau_{N}-D(\Lambda^{s_1}u_{N})-\Delta\Lambda^{s_1}\tau_{N}=\Lambda^{s_1}G.\\[1ex]
\end{array}\right.
\end{align}
Using \eqref{eq8}, Lemmas \ref{Lemma1}-\ref{Lemma2} and time decay rates, we infer that
\begin{align}\label{hineq17}
	&\frac 1 2\frac d {dt} \|\Lambda^s(u_N,\tau_N)\|^2_{L^2}+ \|\nabla\Lambda^s\tau_N\|^2_{L^2}  \\ \notag
	&=-\langle [\Lambda^s,u\cdot\nabla] u,\Lambda^s u_N\rangle-\langle u\cdot\nabla\Lambda^s u,\Lambda^s u_N\rangle+\langle \Lambda^s G,\Lambda^s \tau_N\rangle \\ \notag
	&=-\langle [\Lambda^s,u\cdot\nabla] u,\Lambda^s u_N\rangle-\langle u\cdot\nabla\Lambda^s u_L,\Lambda^s u_N\rangle+\langle \Lambda^{s-1} G,\Lambda^{s+1} \tau_N\rangle \\ \notag
	&\leq C\delta((1+t)^{-2-s}+(1+t)^{-1-\frac s 2}\|\Lambda^{s+1}\tau_N\|_{L^2})+C\|u\|_{L^\infty}\|\nabla\Lambda^{s} u_L\|_{L^2}\|\Lambda^{s} u_N\|_{L^2}  \\ \notag
	&\leq C\delta(1+t)^{-2-s}+C\delta\|\Lambda^{s+1}\tau_N\|^2_{L^2},
	\end{align}
	and
	\begin{align}\label{hineq18}
	&\frac d {dt} \langle \Lambda^{s-1}\tau_N, -\nabla\Lambda^{s-1} u_N\rangle+ \frac 1 2\|\nabla\Lambda^{s-1} u_N\|^2_{L^2}  \\ \notag
	&=\langle \Lambda^{s-1}(G+\Delta \tau_N),-\nabla\Lambda^{s-1} u_N\rangle+\langle \Lambda^{s-1}\mathbb{P}(F+div~\tau_N),div~ \Lambda^{s-1} \tau_N\rangle  \\ \notag
	&\leq \frac 1 4\|\nabla\Lambda^{s-1} u_N\|^2_{L^2}+C\|\Lambda^{s} \tau_N\|^2_{H^1}+C\delta(1+t)^{-1-s}.
	\end{align}	
	Combining \eqref{hineq17} and \eqref{hineq18}, we get
	\begin{align}\label{hineq19}
	&\frac d {dt} [(1+t)\|\Lambda^s(u_N,\tau_N)\|^2_{L^2}+4\langle \Lambda^{s-1}\tau_N, -\nabla\Lambda^{s-1} u_N\rangle]+(1+t)\|\nabla\Lambda^s\tau_N\|^2_{L^2}+\|\nabla\Lambda^{s-1} u_N\|^2_{L^2} \\ \notag
	&\leq \|\Lambda^s(u_N,\tau_N)\|^2_{L^2}+C\|\Lambda^s\tau_N\|^2_{L^2}+C\delta(1+t)^{-1-s},
	\end{align}
	which implies that
	\begin{align}\label{hineq20}
	&\frac d {dt} [(1+t)\|\Lambda^s(u_N,\tau_N)\|^2_{L^2}+4\langle \Lambda^{s-1}\tau_N, -\nabla\Lambda^{s-1} u_N\rangle]+\frac {C_2} {2}\|\Lambda^s\tau_N\|^2_{L^2}+\|\nabla\Lambda^{s-1} u_N\|^2_{L^2} \\ \notag
	&\leq C_2\int_{S_0(t)}|\xi|^{2s}|\hat{\tau}_N(\xi)|^2 d\xi+C\delta(1+t)^{-1-s} \\ \notag
	&\leq C[\delta+(\delta C_1+k_0)C_2](1+t)^{-1-s}.
	\end{align}
	Note that $k_0C_1$ and $C_2$ large enough. We need to take $(\delta C_1+k_0)C_2$ small enough.
	According to \eqref{hineq20}, we obtain
	\begin{align}\label{hineq21}
	(1+t)^{s+2}\|\Lambda^s(u_N,\tau_N)\|^2_{L^2}&\leq 4(1+t)^{s+1}\langle \Lambda^{s-1}\tau_N, \nabla\Lambda^{s-1} u_N\rangle+C[\delta+(\delta+k_0)kC_2](1+t)  \\ \notag
	&\leq \frac 1 2 (1+t)^{s+2}\|\Lambda^s(u_N,\tau_N)\|^2_{L^2}+C[\delta +(\delta C_1+k_0)C_2](1+t).
	\end{align}
	Take suitable constants in \eqref{hineq21}, we have
	\begin{align}\label{hineq22}
	\|\Lambda^s(u_N,\tau_N)\|^2_{L^2}\leq \frac {C^2_0}{4}(1+t)^{-1-s},
	\end{align}
    which implies that
    \begin{align}\label{hineq23}
    \|\Lambda^{s_1}(u_N,\tau_N)\|^2_{L^2}\leq \frac {C^2_0}{4}(1+t)^{-1-s_1}.
    \end{align}
    According to \eqref{hineq10} and \eqref{hineq23}, we infer that
    \begin{align}\label{hineq24}
    \|(u,\tau)\|_{\dot{H}^{s_1}}^2\geq \frac {C^2_0}{4}(1+t)^{-1-s_1}.
    \end{align}	
    Therefore, we complete the proof of Theorem \ref{th2}.
\hfill$\Box$
\section{Optimal decay rate with $\frac{1}{2}\leq \beta<1$}
In this section, we present optimal decay rate for the generalized Oldroyd-B model \eqref{eq1} in fractional case $\frac 1 2\leq\beta<1$. We point out that the refinement of Schonbek's strategy also work in this case with a extra lower energy dissipation estimate.
Similarly, we obtain the decay rates in $H^1$. Moreover, considering the decay rate for the highest derivative of the solution to \eqref{eq1}, the main difficulty is unclosed energy estimate. To overcome this difficulty, we introduce another method which flexibly combines the Fourier spiltting method and the time weighted energy estimate. We rewrite \eqref{eq1} as follows.
\begin{align}\label{eq9}
	\left\{\begin{array}{l}
		\partial_tu +\nabla P - div~\tau=F,~~~~div~u=0,\\[1ex]
		\partial_t\tau +(-\Delta)^{\beta}\tau-D(u)=G.\\[1ex]
	\end{array}\right.
\end{align}
For convenience of explanation, we assume that $\beta=\frac 1 2$. By Theorem \ref{th1}, then we have
	\begin{align}\label{estimate1}
		\frac{d}{dt}(\|(u,\tau)\|^2_{H^s}+2k\langle-\nabla u,\tau\rangle_{H^{s-\frac 1 2}})  + \frac k 2\|\nabla u\|^2_{H^{s-\frac 1 2}} + \|\Lambda^{\frac 1 2}\tau\|^2_{H^s} \leq 0.
	\end{align}
To present optimal decay rate for \eqref{eq9}, we need additional lower energy dissipation estimate for $u$. According to properties of Calderon-Zygmund operator and Lemma \ref{Lemma1}, we infer that
\begin{align}\label{4ineq1}
\frac d {dt} \langle \tau, \Lambda^{-1}(-\nabla u)\rangle+ \frac 1 2\|\Lambda^{\frac 1 2} u\|^2_{L^2}
&=\langle G-\Lambda \tau,-\Lambda^{-1}\nabla u\rangle-\langle \mathbb{P}(F+div~\tau),\Lambda^{-1}div~\tau\rangle  \\ \notag
&\leq C(\|\Lambda^{\frac 1 2} \tau\|_{L^2}\|\Lambda^{\frac 1 2} u\|_{L^2}+\|\Lambda^{\frac 1 2} \tau\|^2_{L^2}  \\ \notag
&+\|u\|^2_{L^4}\|\nabla \tau\|_{L^2}+\|\nabla u\|_{L^2}\|u\|_{L^4}\|\tau\|_{L^4})  \\ \notag
&\leq C\delta\|\Lambda^{\frac 1 2} u\|^2_{L^2}+C\|\Lambda^{\frac 1 2} \tau\|^2_{L^2}.
\end{align}
Combining \eqref{estimate1} and \eqref{4ineq1}, we obtain
\begin{align}\label{estimate2}
\frac{d}{dt}[\|(u,\tau)\|^2_{H^s}+2k(\langle-\nabla u,\tau\rangle_{H^{s-\frac 1 2}}+\langle \tau, -\Lambda^{-1}\nabla u\rangle)] + \frac k 2\|\Lambda^{\frac 1 2} u\|^2_{H^{s}} + \|\Lambda^{\frac 1 2}\tau\|^2_{H^s} \leq 0.
\end{align}
Define $S_1(t)=\{\xi:|\xi|\leq C_2(1+t)^{-1}\}$ for $\beta=\frac 1 2$.
By \eqref{estimate2}, we get
\begin{align}\label{4ineq2}
&\frac d {dt} [\|(u,\tau)\|^2_{H^s}+2k(\langle-\nabla u,\tau\rangle_{H^{s-\frac 1 2}}+\langle \tau, -\Lambda^{-1}\nabla u\rangle)]+\frac {kC_2} {2(1+t)}\| u\|^2_{H^s}+\frac {C_2} {1+t}\|\tau\|^2_{H^{s}}  \\ \notag
&\leq \frac {C} {1+t}\int_{S_1(t)}|\hat{u}(\xi)|^2+|\hat{\tau}(\xi)|^2 d\xi.
\end{align}
According to the refinement of the Fourier spiltting method in Theorem \ref{th2}, one can rediscover the following proposition with $\beta=\frac 1 2$ , which contain optimal decay for the solution in $H^1$ \cite{Wu1}.
\begin{prop}\label{prop5}
	Let $d=2$ and $\frac 1 2\leq\beta<1$. Let $(u,\tau)$ be a strong solution of \eqref{eq0} with the initial data $(u_0,\tau_0)$ under the condition in Theorem \ref{th1}. Then we have
	\begin{align}\label{estimate3}
	\frac{d}{dt}[\|(u,\tau)\|^2_{H^s}+2k(\langle-\nabla u,\tau\rangle_{H^{s-\beta}}+\langle \tau, -\Lambda^{-2+2\beta}\nabla u\rangle)] + \frac k 2\|\Lambda^{\beta} u\|^2_{H^{s+1-2\beta}} + \|\Lambda^{\beta}\tau\|^2_{H^s} \leq 0.
	\end{align}
	In addition, if $(u_0,\tau_0)\in \dot{B}^{-1}_{2,\infty}$, then there exists $C>0$ such that for every $t>0$, there holds
	\begin{align}\label{decay}
	\|\Lambda^{s_2}(u,\tau)\|_{H^{s-s_2}} \leq C(1+t)^{-\frac {1}{2\beta}-\frac{s_2}{2\beta}},
	\end{align}
	where $0\leq s_2\leq 1$.
\end{prop}
We omit the proof of Proposition \ref{prop5}, which is similar to the proof of \eqref{4ineq2} and Theorem \ref{th2}.  \\
{\bf The proof of Theorem \ref{th3}:}
Applying $\Lambda^{s_1}$ to \eqref{eq9}, we obtain
\begin{align}\label{eq10}
\left\{\begin{array}{l}
\partial_t\Lambda^{s_1}u+\nabla\Lambda^{s_1} P-div~\Lambda^{s_1}\tau=\Lambda^{s_1}F,\\[1ex]
\partial_t\Lambda^{s_1}\tau-D(\Lambda^{s_1}u)+(-\Delta)^{\beta}\Lambda^{s_1}\tau=\Lambda^{s_1}G.\\[1ex]
\end{array}\right.
\end{align}
	We first introduce the energy and dissipation functionals for $(u,\tau)$ as follows:
	$$\overline{E}_{\beta}=(1+t)^a\|\Lambda^s(u,\tau)\|^2_{L^2}+k\langle \Lambda^{s-\beta}\tau, -\nabla\Lambda^{s-\beta} u\rangle,$$
	and
	$$\overline{D}_{\beta}=(1+t)^a\|\Lambda^{s+\beta}\tau\|^2_{L^2}+ \frac k 4\|\nabla\Lambda^{s-\beta} u\|^2_{L^2},$$
	where $a=2-\frac 1 {\beta}\in[0,1)$ and $k$ is small enough to overcome the difficulty about inner product estimate of linear term $div\tau$ in $(\ref{eq10})_1$ with $\beta=\frac 1 2$. By \eqref{eq10}, we have
	\begin{align}\label{4ineq3}
	\frac 1 2\frac d {dt} \|\Lambda^s(u,\tau)\|^2_{L^2}+ \|\Lambda^{s+\beta}\tau\|^2_{L^2}=\langle \Lambda^s F,\Lambda^s u\rangle+\langle \Lambda^s G,\Lambda^s \tau\rangle,
	\end{align}
	and
	\begin{align}\label{4ineq4}
	&\frac d {dt} \langle \Lambda^{s-\beta}\tau, -\nabla\Lambda^{s-\beta} u\rangle+ \frac 1 2\|\nabla\Lambda^{s-\beta} u\|^2_{L^2}  \\ \notag
	&=\langle (-\Lambda^{s-\beta}G-\Lambda^{s+\beta} \tau),\nabla\Lambda^{s-\beta} u\rangle+\langle \Lambda^{s-\beta}\mathbb{P}(F+div~\tau),div~\Lambda^{s-\beta} \tau\rangle.
	\end{align}	
	Combining \eqref{4ineq3} and \eqref{4ineq4}, we get
	\begin{align}\label{4ineq5}
    \frac d {dt} \overline{E}_{\beta}+ 2\overline{D}_{\beta}
	&=a(1+t)^{a-1}\|\Lambda^s(u,\tau)\|^2_{L^2}+2(1+t)^a(\langle \Lambda^s F,\Lambda^s u\rangle+\langle \Lambda^s G,\Lambda^s \tau\rangle) \\ \notag
	&+k\langle (-\Lambda^{s-\beta}G-\Lambda^{s+\beta} \tau),\nabla\Lambda^{s-\beta} u\rangle+k\langle \Lambda^{s-\beta}\mathbb{P}(F+div~\tau),div~\Lambda^{s-\beta} \tau\rangle.
    \end{align}	
	By Lemma \ref{Lemma1}, we obtain $\|\Lambda^{s} u\|_{L^2}\leq C\|u\|^{1-\theta}_{L^2}\|\Lambda^{s+1-\beta} u\|^{\theta}_{L^2}$ and $\|\nabla u\|_{L^\infty}\leq C\| u\|^{1-\theta_1}_{L^2}\|\Lambda^{s+1-\beta} u\|^{\theta_1}_{L^2}$ with $(\theta,\theta_1)=(\frac {s} {s+1-\beta},\frac {2} {s+1-\beta})$. By Lemmas \ref{Lemma1}-\ref{Lemma2} and Proposition \ref{prop5}, we have
	\begin{align}\label{4ineq6}
	(1+t)^a\langle \Lambda^s F,\Lambda^s u\rangle &\leq C(1+t)^a\|\nabla u\|_{L^\infty}\|\Lambda^s u\|^2_{L^2} \\ \notag
	&\leq C(1+t)^a\|\nabla u\|^{\beta}_{L^\infty}\|\nabla u\|^{1-\beta}_{L^2}\|\Lambda^{s+1-\beta} u\|^2_{L^2} \\ \notag
	&\leq C(1+t)^a(1+t)^{-\frac {1}{2}-\frac {1}{2\beta}}\|\Lambda^{s+1-\beta} u\|^2_{L^2} \\ \notag
	&\leq C(1+t)^{\frac 3 2-\frac {3}{2\beta}}\|\Lambda^{s+1-\beta} u\|^2_{L^2}.
	\end{align}	
	When $\frac 1 2\leq\beta<1$, we have $(1+t)^a\langle \Lambda^s F,\Lambda^s u\rangle \leq \delta \overline{D}_{\beta}$.
	Similarly, we deduce that
	\begin{align}\label{4ineq7}
	(1+t)^a\langle \Lambda^s G,\Lambda^s \tau\rangle &\leq C(1+t)^a\|\Lambda^{s-\beta} G\|_{L^2}\|\Lambda^{s+\beta} \tau\|_{L^2} \\ \notag
	&\leq C(1+t)^a\|\Lambda^{s+\beta} \tau\|_{L^2}(\| u\|_{L^\infty}\|\Lambda^{s+1-\beta} \tau\|_{L^2}+\|\tau\|_{L^\infty}\|\nabla\Lambda^{s-\beta} u\|_{L^2} ) \\ \notag
	&\leq \delta \overline{D}_{\beta}+C(1+t)^a(\| u\|^2_{L^\infty}\|\Lambda^{s+1-\beta} \tau\|^2_{L^2}+\|\tau\|^2_{L^\infty}\|\nabla\Lambda^{s-\beta} u\|^2_{L^2}) \\ \notag
	&\leq 2\delta \overline{D}_{\beta}+(1+t)^{a-\frac 1 {\beta}}\|\Lambda^{s+1-\beta} \tau\|^2_{L^2}.
	\end{align}
	Applying Lemma \ref{Lemma1} again, we get $\|\Lambda^{s+1-\beta}  \tau\|_{L^2}\leq C\| \tau\|^{1-\theta_2}_{L^2}\|\Lambda^{s+\beta} \tau\|^{\theta_2}_{L^2}$ with $\theta_2=\frac {s+1-\beta} {s+\beta}$, which implies that
	\begin{align}\label{4ineq8}
	\|\Lambda^{s+1-\beta} \tau\|^2_{L^2}&\leq C(1+t)^{-\frac {s+1} {\beta}+1}\|\tau\|^2_{L^2}+C(1+t)^a\|\Lambda^{s+\beta} \tau\|^{2}_{L^2}	\\ \notag
	&\leq C(1+t)^{-\frac {s+1} {\beta}+a-1}+C(1+t)^a\|\Lambda^{s+\beta} \tau\|^{2}_{L^2}.
	\end{align}
	Combining \eqref{4ineq7} and \eqref{4ineq8}, we infer that
	\begin{align}\label{4ineq9}
	(1+t)^a\langle \Lambda^s G,\Lambda^s \tau\rangle \leq C\delta \overline{D}_{\beta}+C(1+t)^{-\frac {s+1} {\beta}+a-1}.
	\end{align}
    Moreover, we can easily deduce from \eqref{4ineq7} and \eqref{4ineq8} that
    \begin{align}\label{4ineq10}
    &k\langle (-\Lambda^{s-\beta}G-\Lambda^{s+\beta} \tau),\nabla\Lambda^{s-\beta} u\rangle+k\langle \Lambda^{s-\beta}\mathbb{P}(F+div \tau),div \Lambda^{s-\beta} \tau\rangle  \\ \notag
    &\leq C(\delta+k)\overline{D}_{\beta}+C(1+t)^{-\frac {s+1} {\beta}+a-1}.
    \end{align}
    According \eqref{4ineq5}-\eqref{4ineq10}, we infer that
	\begin{align}\label{4ineq11}
	\frac d {dt} \overline{E}_{\beta}+ 2\overline{D}_{\beta}
	&\leq a(1+t)^{a-1}\|\Lambda^s(u,\tau)\|^2_{L^2}+C(\delta+k) \overline{D}_{\beta}++C(1+t)^{-\frac {s+1} {\beta}+a-1}.
	\end{align}
    Define $S^{\beta}(t)=\{\xi:|\xi|^{2\beta}\leq 4C_2(1+t)^{-1}\}$.
    By \eqref{4ineq11}, we get
	\begin{align}\label{4ineq12}
	&\frac d {dt} \overline{E}_{\beta}+ kC_2(1+t)^{a-1}\|\Lambda^s(u,\tau)\|^2_{L^2}  \\ \notag
	&\leq C(1+t)^{-\frac {s+1} {\beta}+a-1}+C(1+t)^{a-1}
	\int_{S^{\beta}(t)}|\xi|^{2s}(|\hat{u}(\xi)|^2+|\hat{\tau}(\xi)|^2) d\xi \\ \notag
	&\leq C(1+t)^{-\frac {s+1} {\beta}+a-1},
	\end{align}	
	which implies that
	\begin{align}\label{4ineq13}
	&(1+t)^{\frac {s+1} {\beta}-a+1} \overline{E}_{\beta}+\int_{0}^{t}kC_2(1+s')^{\frac {s+1} {\beta}} \|\Lambda^s(\tau, u)\|^2_{L^2}ds'  \\  \notag
	&\leq C(1+t)+C_s\int_{0}^{t}(1+s')^{\frac {s+1} {\beta}-a}\langle \Lambda^{s-\beta}\tau, -\nabla\Lambda^{s-\beta} u\rangle ds' \\ \notag
	&\leq C(1+t)+C_s\int_{0}^{t}(1+s')^{\frac {s+1} {\beta}} \|\Lambda^{s} u\|^2_{L^2}ds'+C_s\int_{0}^{t}(1+s')^{\frac {s+1} {\beta}-2a} \|\Lambda^{s+1-2\beta} \tau\|^2_{L^2}ds'  \\ \notag
	&\leq C(1+t)+C_s\int_{0}^{t}(1+s')^{\frac {s+1} {\beta}} \|\Lambda^s(\tau, u)\|^2_{L^2}ds'+C_s\int_{0}^{t}(1+s')^{\frac {1-s} {\beta}} \|\tau\|^2_{L^2}ds'  \\ \notag
	&\leq C(1+t)+C_s\int_{0}^{t}(1+s')^{\frac {s+1} {\beta}} \|\Lambda^s(\tau, u)\|^2_{L^2}ds'.
	\end{align}	
	Taking $kC_2\geq C_s$ and using Lemma \ref{Lemma1}, we infer that
	\begin{align}\label{4ineq14}
	(1+t)^{\frac {s+1} {\beta}+1}\|\Lambda^s(u,\tau)\|^2_{L^2}&\leq C(1+t)+Ck(1+t)^{\frac {s+1} {\beta}-a+1}\langle \Lambda^{s-\beta}\tau, -\nabla\Lambda^{s-\beta} u\rangle \\  \notag
	&\leq C[(1+t)+k(1+t)^{\frac {s+1+\beta} {\beta}}\|\Lambda^s u\|^2_{L^2}+k(1+t)^{\frac {s+3-3\beta} {\beta}}\|\Lambda^{s+1-2\beta}\tau\|^2_{L^2}] \\  \notag
	&\leq C(1+t)+Ck(1+t)^{\frac {s+1} {\beta}+1}\|\Lambda^s(u,\tau)\|^2_{L^2}+Ck(1+t)^{\frac {1-s}{\beta}+1}\|\tau\|^2_{L^2},
	\end{align}
	which implies that
	\begin{align}\label{4ineq15}
	\|\Lambda^s(u,\tau)\|^2_{L^2}\leq C(1+t)^{-\frac {s+1} {\beta}}.
	\end{align}	
	The lower bound of time decay rate can be obtained by the similar method in Theorem \ref{th2}. We omit the details here.
\hfill$\Box$
	
\smallskip
\noindent\textbf{Acknowledgments} This work was
partially supported by the National Natural Science Foundation of China (No.12171493 and No.11671407), the Macao Science and Technology Development Fund (No. 0091/2018/A3), Guangdong Province of China Special Support Program (No. 8-2015),
the key project of the Natural Science Foundation of Guangdong province (No. 2016A030311004), and National Key R$\&$D Program of China (No. 2021YFA1002100).


\phantomsection
\addcontentsline{toc}{section}{\refname}
\bibliographystyle{abbrv} 
\bibliography{OldroydBref}

\end{document}